\documentclass[11pt]{amsart}

\usepackage[T1]{fontenc}
\usepackage{lmodern}
\usepackage{microtype}
\usepackage{amssymb}
\usepackage{mathtools}
\usepackage[sort,compress]{cite}
\usepackage{url}
\usepackage[hidelinks]{hyperref}

\hypersetup{
  pdftitle={Central Splitting, Division-Stable Residuals, and Opposite-Ring Transfer for Strongly C4*-Rings},
  pdfauthor={Chandrasekhar Gokavarapu, Naveen Kumar Kakumanu, Sajani Lavanya Madasi, Sudakar Gadde, and Siri Katakam},
  pdfsubject={Rings and Algebras},
  pdfkeywords={C4*-ring, strongly C4*-ring, central idempotent, opposite ring, division-stable residual}
}

\theoremstyle{plain}
\newtheorem{theorem}{Theorem}[section]
\newtheorem{proposition}[theorem]{Proposition}
\newtheorem{lemma}[theorem]{Lemma}
\newtheorem{corollary}[theorem]{Corollary}

\theoremstyle{definition}
\newtheorem{definition}[theorem]{Definition}
\newtheorem{example}[theorem]{Example}

\theoremstyle{remark}
\newtheorem{remark}[theorem]{Remark}

\numberwithin{equation}{section}

\newcommand{\Cfour}{\mathrm{C4}}
\newcommand{\Cfourstar}{\mathrm{C4}^{\ast}}
\newcommand{\Hom}{\operatorname{Hom}}

\newcommand{\op}{\mathrm{op}}
\newcommand{\ess}{\subseteq_{\mathrm{ess}}}
\newcommand{\ds}{\subseteq_{\oplus}}

\newcommand{\id}{\mathrm{id}}
\newcommand{\Idem}{\operatorname{Idem}}
\newcommand{\Ecal}{\mathfrak E}

\begin{document}

\title[Division-Stable Residuals and Opposite-Ring Transfer]
{Central Splitting, Division-Stable Residuals, and Opposite-Ring
Transfer for Strongly \(\Cfourstar\)-Rings}

\author{Chandrasekhar Gokavarapu}
\email{chandrasekhargokavarapu@gmail.com}
\email{chandrasekhargokavarapu@gcrjy.ac.in}
\thanks{Corresponding author: Chandrasekhar Gokavarapu.}

\author{Naveen Kumar Kakumanu}

\author{Sajani Lavanya Madasi}

\author{Sudakar Gadde}
\address{Chandrasekhar Gokavarapu, Naveen Kumar Kakumanu,
Sajani Lavanya Madasi, and Sudakar Gadde:
Department of Mathematics, Government College (Autonomous),
Rajahmundry, Andhra Pradesh 533105, India}

\author{Siri Katakam}
\address{Siri Katakam:
M.Sc.\ Mathematics Programme, 2025--2027 Batch,
Government College (Autonomous), Rajahmundry,
Andhra Pradesh 533105, India}

\renewcommand{\shortauthors}{Chandrasekhar Gokavarapu et al.}

\date{}

\begin{abstract}
The known decomposition theorem for strongly \(\Cfourstar\)-modules gives a
semisimple summand and a summand-square-free residual summand, with
two-way Hom-orthogonality when the ambient module is projective.  Applied
to the regular module, the two cross-corners vanish, so the idempotent
defining the decomposition is central.  Thus every strongly right
\(\Cfourstar\)-ring splits as \(R\cong \Sigma\times T\),
where \(\Sigma\) is semisimple artinian and \(T_T\) is
summand-square-free.

The splitting need not be unique.  We organize all admissible central
splittings into a join-semilattice and prove a comparison theorem: any
two residual factors have a common direct factor, and their complementary
factors are finite products of division rings.  Consequently the
right-to-left defect is independent of the chosen splitting.  If the
central idempotents satisfy the ascending chain condition, there is a
unique greatest admissible idempotent; it captures every central
semisimple artinian direct factor and hence yields a canonical residual
with no further such factor.  An infinite product of fields shows that
this finiteness hypothesis cannot simply be omitted.

We also prove anti-isomorphism transport for \(\Cfour\),
\(\Cfourstar\), semi-weak-CS, and strongly \(\Cfourstar\) modules.
It yields transfer whenever a residual factor is a product of a
semisimple summand-square-free ring and a self-opposite core.  This
criterion does not assume regularity, exchange, or a left-sided
hypothesis.  Skew Laurent rings provide noncommutative nonregular
examples, while the known injective-nonsurjective skew-polynomial
construction gives a sharp one-sided residual obstruction.
\end{abstract}

\keywords{\(\Cfourstar\)-ring; strongly \(\Cfourstar\)-ring;
central idempotent; opposite ring; anti-isomorphism; square-free module;
division ring; skew Laurent ring.}

\subjclass[2020]{16D70, 16D80, 16W10, 16S34}

\maketitle

\section{Introduction}

The class of \(\Cfour\)-modules was introduced as a common extension of
\(\mathrm{C3}\)-modules and square-free modules
\cite{DingIbrahimYousifZhou2017,AltunIbrahimOzcanYousif2018}.  A module is
\(\Cfourstar\) when all its submodules are \(\Cfour\).  The strong
\(\Cfourstar\) condition adds the semi-weak-CS property
\cite{IbrahimEidElGuindy2026}.  This strong class admits a decomposition
into a semisimple part and a summand-square-free part.  For regular rings
the residual part becomes strongly regular; full left--right symmetry
then follows \cite{IbrahimEidElGuindy2026}.

Regularity is not a dispensable technical assumption in that argument.
There are skew polynomial domains which are strongly \(\Cfourstar\) on
one side and not on the other
\cite[Example~4.12]{IbrahimEidElGuindy2026}.  A theorem beyond the
regular case must therefore identify a genuine symmetry mechanism and
locate exactly where one-sided behavior can survive.

\subsection{Related literature and positioning}

The conditions considered here belong to the direct-summand tradition
originating in continuous and quasi-continuous module theory
\cite{MohamedMuller1990,Lam1999}.  General decomposition methods and the
role of endomorphism rings are developed systematically in
\cite{Facchini1998,Facchini2012}.  Within this tradition, the
\(\mathrm{C3}\) and \(\mathrm{D3}\) conditions were studied in
\cite{EbrahimiAtaniKhoramdelHesari2016,YousifAminIbrahim2014}, and their
homomorphism-based extensions led to \(\Cfour\)- and
\(\mathrm{D4}\)-modules
\cite{DingIbrahimYousifZhou2017,DingIbrahimYousifZhouD42017}.
Perspective direct summands give a further description of these classes
\cite{AltunIbrahimOzcanYousif2018}, while class-relative variants of
\(\mathrm{C3}\), \(\Cfour\), and \(\mathrm{D4}\) have appeared more
recently \cite{Zhu2023,DialloDiopKourkiTribak2025}.  The
\(\Cfourstar\), semi-weak-CS, and strongly \(\Cfourstar\) notions used
below are the latest structural refinement in this line
\cite{IbrahimEidElGuindy2026}.

Two earlier decomposition results are particularly close in form but
not in content to the present theorem.  Rings all of whose right ideals
are \(U\)-modules split into a square-full semisimple artinian part and a
right square-free part \cite{IbrahimYousif2018}.  Likewise, a right
\(a\)-ring, whose right ideals are automorphism-invariant, has a
square-full semisimple artinian factor and a right square-free factor
\cite{KosanQuynhSrivastava2016}.  For a strongly \(\Cfourstar\)-module,
however, the available source decomposition has a semisimple summand and
only a \emph{summand}-square-free residual, together with two-way
Hom-orthogonality when the module is projective.  The present paper uses
that orthogonality on \(R_R\) to prove centrality of the splitting
idempotent.  It then studies the entire family of such central
splittings, proves comparison modulo finite products of division rings,
and obtains a greatest splitting under ACC on central idempotents.
These semilattice, comparison, and canonical-residual conclusions do not
follow from the cited \(U\)-module or automorphism-invariant
decompositions.

The question of comparing decompositions has a long history, from
refinement theory \cite{CrawleyJonsson1964} to modern treatments of
modules with semilocal endomorphism rings
\cite{Facchini1998,Facchini2012}.  A related but different direction
studies common complements and perspectivity
\cite{GargGroverKhurana2014}.  Strong perspectivity, exchange, and
summand-dual-square-free modules were investigated in
\cite{IbrahimYousifPerspective2022}, and the near-uniqueness of direct
complements was developed in
\cite{IbrahimYousifDirectComplements2023}.  Our comparison theorem is not
a cancellation or common-complement theorem: it compares two
\emph{central ring factors} and identifies the exact discrepancy as
finite products of division rings.

Exchange theory supplies another important neighboring framework.
Warfield's module-theoretic formulation, Nicholson's idempotent-lifting
characterization, and the clean-ring results of Camillo and Yu form the
classical background
\cite{Warfield1972,Nicholson1977,CamilloYu1994}.  Automorphism-invariant
modules also satisfy exchange and cleanness properties
\cite{GuilAsensioSrivastava2013}.  For \(\Cfour\)-modules, restricted
chain conditions connect finite exchange, full exchange, and cleanness
\cite{IbrahimYousif2022}.  Square-free and dual-square-free versions of
these phenomena are studied in
\cite{Nielsen2010,MazurekNielsenZiembowski2015,IbrahimYousif2019}.  By
contrast, our central splitting and residual comparison require no
exchange or cleanness hypothesis.  The ACC used later is imposed only on
central idempotents and only to select a greatest admissible splitting;
it is not used to establish an exchange property.

Finally, the side-changing part of the paper is distinct from
left--right symmetric perspectivity.  It transports the complete
\(\Cfourstar\) and semi-weak-CS definitions through an
anti-isomorphism, in the standard opposite-ring setting of
\cite{Lam2001,KnusMerkurjevRostTignol1998,Pierce1982}.  The nonregular
examples lie in the Ore-extension tradition
\cite{Ore1933,GoodearlWarfield2004,McConnellRobson2001}.  This separates
two issues that are sometimes conflated: symmetry coming from a
side-changing anti-isomorphism, and one-sided failure caused by a
nonsurjective twisting endomorphism.

An earlier preprint \cite{GokavarapuTransfer2026} formulated the transfer
program using auxiliary corner and annihilator hypotheses.  The present
paper supersedes that framework: the central splitting follows directly
from the two Hom-vanishings in the source decomposition theorem, and the
new results concern the semilattice, comparison, and canonicality of all
such splittings.  A separate preprint
\cite{GokavarapuMorita2026} studies covariant Morita equivalence and
matrix/full-corner permanence.  Those results are not repeated here.
The anti-isomorphism argument below is side-changing and is used only as
a tool for the residual problem.

The first mechanism in this paper is central splitting.
The decomposition theorem of \cite{IbrahimEidElGuindy2026}, when applied
to \(R_R\), supplies
\[
R_R=P\oplus Q,
\]
with \(P\) semisimple and \(Q\) summand-square-free.  Since \(R_R\) is
projective, the same theorem gives
\[
\Hom_R(P,Q)=0=\Hom_R(Q,P).
\]
Writing \(P=eR\) and \(Q=(1-e)R\), these two equalities are precisely
\[
(1-e)Re=0=eR(1-e).
\]
Consequently \(e\) is central.  The module decomposition is therefore a
ring decomposition.  This yields
\[
R\cong eR\times(1-e)R
\]
with \(eR\) semisimple artinian and \((1-e)R\) the residual
summand-square-free ring.  No annihilator closure and no additional
splitting axiom is required.

Centrality raises a question that is invisible at the level of a single
decomposition: how different can two residual factors be?  For a
strongly right \(\Cfourstar\)-ring \(R\), let \(\Ecal(R)\) be the set of
central idempotents \(e\) for which \(eR\) is semisimple artinian and
\((1-e)R\) is summand-square-free.  We prove that \(\Ecal(R)\) is closed
under finite joins.  For \(e,f\in\Ecal(R)\), the four central corners give
\[
\begin{aligned}
(1-e)R&\cong (1-e)fR\times(1-e)(1-f)R,\\
(1-f)R&\cong e(1-f)R\times(1-e)(1-f)R.
\end{aligned}
\]
The two unequal factors are simultaneously semisimple and
summand-square-free, hence finite products of division rings.  Thus
residuals are unique modulo division-ring factors.  Under ACC on central
idempotents, finite-join closure produces a unique greatest admissible
idempotent.  The resulting residual is canonical, and every other
residual is its product with finitely many division rings.

The second mechanism is opposite-ring transport.  An anti-isomorphism
\(\alpha:R\to S\) identifies the submodule lattice of a right
\(R\)-module with the submodule lattice of a left \(S\)-module.  It
preserves kernels, images, direct summands, essential submodules,
semisimplicity, and the chains occurring in the semi-weak-CS condition.
We prove that it therefore preserves and reflects the \(\Cfour\),
\(\Cfourstar\), semi-weak-CS, and strongly \(\Cfourstar\) properties.
At ring level,
\[
R\text{ strongly right }\Cfourstar
\quad\Longleftrightarrow\quad
S\text{ strongly left }\Cfourstar.
\]

Combining the comparison and transport mechanisms gives a criterion
strictly more flexible than self-oppositeness of an entire residual.  If
\[
T\cong D\times U,
\]
where \(D\) is a finite product of division rings and
\(U\cong U^{\op}\), then \(T\), and hence \(R\), transfers from right to
left.  No anti-isomorphism is required on \(D\).  Skew Laurent rings
provide nonregular choices for \(U\).

The paper also gives an exact reduction of counterexamples.  If \(R\) is
strongly right but not strongly left \(\Cfourstar\), every admissible
residual has the same defect.  The comparison theorem shows that this is
independent of the decomposition, not merely true one splitting at a
time.

Section~\ref{sec:preliminaries} fixes the definitions and input from the
existing theory.  Section~\ref{sec:central} proves central splitting.
Section~\ref{sec:comparison} develops admissible splittings, comparison,
and canonicality.  Section~\ref{sec:anti} establishes side-changing
transport.  Section~\ref{sec:transfer} proves the division-stable
transfer criterion.  Section~\ref{sec:examples} gives nonregular
families, and Section~\ref{sec:obstruction} locates one-sided defects.

\section{Preliminaries}\label{sec:preliminaries}

All rings are associative with identity.  Modules are unitary.  We write
\(N\ds M\) when \(N\) is a direct summand of \(M\), and
\(N\ess M\) when \(N\) is essential in \(M\).  Standard facts concerning
idempotents, direct summands, opposite rings, and skew extensions may be
found in
\cite{Lam1999,Lam2001,GoodearlWarfield2004,McConnellRobson2001}.

\begin{definition}
A right \(R\)-module \(M\) is a \(\Cfour\)-module if, whenever
\[
M=A\oplus B
\]
and \(f:A\to B\) is an \(R\)-homomorphism with
\(\ker f\ds A\), one has \(\operatorname{Im}f\ds B\).
A module \(M\) is a \(\Cfourstar\)-module if every submodule of \(M\) is
a \(\Cfour\)-module.
\end{definition}

\begin{definition}
Let \(M\) be a module.  Denote by \(\mathcal S(M)\) the set of triples
\((X,Y,f)\) in which \(X,Y\ds M\) are semisimple,
\(X\cap Y=0\), and \(f:X\to Y\) is an isomorphism.  The order is
extension in each coordinate.  The module \(M\) is \emph{semi-weak-CS}
if, for every chain in \(\mathcal S(M)\), the unions of its first and
second coordinates are essential in direct summands of \(M\).
It is \emph{strongly \(\Cfourstar\)} if it is both
\(\Cfourstar\) and semi-weak-CS.
\end{definition}

A ring \(R\) is a right \(\Cfourstar\)-ring, respectively a strongly
right \(\Cfourstar\)-ring, when \(R_R\) has the indicated property.
The left-handed notions are defined using \({}_RR\).

\begin{definition}
A module \(M\) is \emph{square-free} if it contains no nonzero submodules
\(A,B\) such that \(A\cap B=0\) and \(A\cong B\).  It is
\emph{summand-square-free} if the same prohibition is imposed only on
direct summands \(A,B\ds M\).
\end{definition}

The distinction is essential: summand-square-freeness is strictly weaker
than square-freeness in general.  Square-free and exchange phenomena are
studied in
\cite{Nielsen2010,MazurekNielsenZiembowski2015}.  We shall use the
following facts from \cite{IbrahimEidElGuindy2026}.

\begin{proposition}\label{prop:squarefree-strong}
Every square-free module is a \(\Cfourstar\)-module and is
semi-weak-CS.  Hence every square-free module is strongly
\(\Cfourstar\).
\end{proposition}

\begin{theorem}[Strong decomposition theorem]\label{thm:source-decomp}
Let \(M\) be a strongly \(\Cfourstar\)-module.  Then
\[
M=P\oplus Q,
\]
where \(P\) is semisimple, \(Q\) is summand-square-free, and
\[
\Hom_R(X,Y)=0
\qquad (X\leq P,\;Y\leq Q).
\]
If \(M\) is projective, then also
\[
\Hom_R(Y,X)=0
\qquad (Y\leq Q,\;X\leq P).
\]
\end{theorem}

\begin{proof}
This is the portion of
\cite[Theorem~3.12]{IbrahimEidElGuindy2026} used below.
\end{proof}

We record a finite-product fact.  Its proof will be used on both sides.

\begin{lemma}[Product lemma]\label{lem:product}
Let \(A\) and \(B\) be rings.
\begin{enumerate}
\item \(A\times B\) is a right \(\Cfourstar\)-ring if and only if both
      \(A\) and \(B\) are right \(\Cfourstar\)-rings.
\item \(A\times B\) is strongly right \(\Cfourstar\) if and only if both
      \(A\) and \(B\) are strongly right \(\Cfourstar\).
\item The analogous assertions hold on the left.
\end{enumerate}
\end{lemma}

\begin{proof}
Put \(c=(1,0)\) and \(d=(0,1)\).  Every \(A\times B\)-module \(M\)
decomposes as \(cM\oplus dM\), and every submodule \(N\leq M\) decomposes
as \(cN\oplus dN\).  Homomorphisms, kernels, images, and direct summands
are therefore computed componentwise.  This proves the
\(\Cfourstar\) assertions.

Semisimple submodules, essentiality, and direct summands are also
computed componentwise.  A chain in \(\mathcal S(M)\) consequently
projects to chains in \(\mathcal S(cM)\) and \(\mathcal S(dM)\), and its
unions are essential in direct summands precisely when the projected
unions have that property.  This proves the semi-weak-CS assertions and
hence the strong assertions.
\end{proof}

\section{Two-way orthogonality and central splitting}
\label{sec:central}

The passage from a module decomposition of \(R_R\) to a ring
decomposition is controlled by two cross-corners.  We first make this
identification explicit.

\begin{lemma}[Corner--Hom identification]\label{lem:corner-hom}
Let \(e^2=e\in R\).  Evaluation at the indicated generator gives
natural isomorphisms of abelian groups
\[
\Hom_R(eR,(1-e)R)\cong(1-e)Re
\]
and
\[
\Hom_R((1-e)R,eR)\cong eR(1-e).
\]
\end{lemma}

\begin{proof}
If \(f:eR\to(1-e)R\), then \(f(e)=f(e^2)=f(e)e\), so
\(f(e)\in(1-e)Re\).  Conversely, \(z\in(1-e)Re\) defines
\[
f_z:eR\longrightarrow(1-e)R,\qquad f_z(er)=zr.
\]
This is well defined because \(z=ze\).  The two assignments are inverse.
The second isomorphism is proved in the same way.
\end{proof}

\begin{lemma}[Centrality criterion]\label{lem:centrality}
Let \(e^2=e\in R\).  If
\[
\Hom_R(eR,(1-e)R)=0=\Hom_R((1-e)R,eR),
\]
then \(e\) is central and
\[
R\cong eR\times(1-e)R
\]
as rings.
\end{lemma}

\begin{proof}
Lemma~\ref{lem:corner-hom} gives
\[
(1-e)Re=0=eR(1-e).
\]
For \(r\in R\),
\[
re=ere+(1-e)re=ere,
\qquad
er=ere+er(1-e)=ere.
\]
Thus \(re=er\).  The two ideals \(eR\) and \((1-e)R\) annihilate each
other and have identities \(e\) and \(1-e\), respectively.  The displayed
ring product follows.
\end{proof}

We now obtain the first principal result.

\begin{theorem}[Central splitting theorem]\label{thm:central-splitting}
Let \(R\) be a strongly right \(\Cfourstar\)-ring.  There is a central
idempotent \(e\in R\) such that
\[
R\cong \Sigma\times T,
\qquad
\Sigma=eR,\quad T=(1-e)R,
\]
and the following hold:
\begin{enumerate}
\item \(\Sigma\) is semisimple artinian;
\item \(T_T\) is summand-square-free;
\item both \(\Sigma\) and \(T\) are strongly right
      \(\Cfourstar\)-rings.
\end{enumerate}
\end{theorem}

\begin{proof}
Apply Theorem~\ref{thm:source-decomp} to the projective module \(R_R\).
We obtain
\[
R_R=P\oplus Q
\]
with \(P\) semisimple, \(Q\) summand-square-free, and
\[
\Hom_R(P,Q)=0=\Hom_R(Q,P).
\]
Since \(P\ds R_R\), there is an idempotent \(e\in R\) with
\[
P=eR,\qquad Q=(1-e)R.
\]
Lemma~\ref{lem:centrality} shows that \(e\) is central and that
\(R\cong eR\times(1-e)R\).

The right ideal \(eR\) is cyclic and semisimple.  It therefore has finite
length.  Since the \(R\)-action on \(eR\) factors through the ring \(eR\),
the right regular module of \(eR\) is semisimple of finite length.
Consequently \(eR\) is semisimple artinian.

The categories of right \(R\)-submodules and right \(T\)-submodules of
\((1-e)R\) coincide.  Direct summands coincide as well.  Hence \(T_T\)
is summand-square-free.  Finally, the product lemma shows that both
factors inherit the strongly right \(\Cfourstar\) property from \(R\).
\end{proof}

\section{Comparison and canonicality of residual factors}
\label{sec:comparison}

The central splitting theorem does not assert uniqueness of \(e\).
This section determines the precise ambiguity.

\begin{definition}\label{def:admissible}
Let \(R\) be a strongly right \(\Cfourstar\)-ring.  Its set of
\emph{admissible splitting idempotents} is
\[
\begin{aligned}
\Ecal(R)=\{e\in Z(R):\;&e^2=e,\quad
eR\text{ is semisimple artinian},\\
&((1-e)R)_{(1-e)R}\text{ is summand-square-free}\}.
\end{aligned}
\]
For \(e\in\Ecal(R)\), put
\[
\Sigma_e=eR,\qquad T_e=(1-e)R.
\]
\end{definition}

Theorem~\ref{thm:central-splitting} says that \(\Ecal(R)\neq\varnothing\).
The next elementary classification will identify the entire ambiguity
between two residual factors.

\begin{lemma}[Semisimple summand-square-free rings]
\label{lem:ssf-semisimple}
Let \(S\) be semisimple artinian.  Then \(S_S\) is
summand-square-free if and only if
\[
S\cong D_1\times\cdots\times D_n
\]
for division rings \(D_1,\ldots,D_n\), with the empty product allowed
when \(S=0\).
\end{lemma}

\begin{proof}
By Wedderburn--Artin,
\[
S\cong\prod_{i=1}^{n}M_{m_i}(D_i).
\]
The right regular module of \(M_{m_i}(D_i)\) is a direct sum of \(m_i\)
isomorphic simple modules.  If some \(m_i\geq2\), two of these simple
summands violate summand-square-freeness.

Conversely, suppose every \(m_i=1\).  Every direct summand of \(S_S\)
has the form \(cS\) for a central idempotent \(c\).  If
\(cS\cap dS=0\), then \(cd=0\), and
\[
\Hom_S(cS,dS)\cong dSc=0.
\]
Thus two nonzero disjoint direct summands cannot be isomorphic.
\end{proof}

We order central idempotents by \(e\leq f\) when \(ef=e\), and write
\[
e\vee f=e+f-ef.
\]

\begin{proposition}[Join and absorption]\label{prop:join}
Let \(e\in\Ecal(R)\).
\begin{enumerate}
\item If \(f\in\Ecal(R)\), then \(e\vee f\in\Ecal(R)\).
\item More generally, if \(c\in Z(R)\) is idempotent and \(cR\) is
      semisimple artinian, then \(e\vee c\in\Ecal(R)\).
\end{enumerate}
\end{proposition}

\begin{proof}
It is enough to prove (2).  Put \(g=e\vee c\).  Since \(e\) and \(c\)
are central,
\[
gR=eR\times(1-e)cR.
\]
The first factor is semisimple artinian, and the second is a direct
factor of the semisimple artinian ring \(cR\).  Hence \(gR\) is
semisimple artinian.  Moreover,
\[
(1-g)R=(1-e)(1-c)R
\]
is a direct summand of \(T_e=(1-e)R\).  A direct summand of a
summand-square-free module is summand-square-free, so \(g\in\Ecal(R)\).
Part (1) follows because \(fR\) is semisimple artinian.
\end{proof}

\begin{theorem}[Division-stable comparison]\label{thm:comparison}
Let \(e,f\in\Ecal(R)\), and set
\[
C_{e,f}=(1-e)(1-f)R,\quad
D_{e,f}=e(1-f)R,\quad
D_{f,e}=f(1-e)R.
\]
Then
\[
T_e\cong D_{f,e}\times C_{e,f},
\qquad
T_f\cong D_{e,f}\times C_{e,f},
\]
and both \(D_{e,f}\) and \(D_{f,e}\) are finite products of division
rings.
\end{theorem}

\begin{proof}
The displayed decompositions are the decompositions of \(T_e\) and
\(T_f\) induced by the central idempotents \(f\) and \(e\), respectively.
The ring \(D_{e,f}\) is a direct factor of \(\Sigma_e=eR\), so it is
semisimple artinian.  It is also a direct factor of \(T_f=(1-f)R\), so
its right regular module is summand-square-free.  By
Lemma~\ref{lem:ssf-semisimple}, it is a finite product of division rings.
The same argument applies to \(D_{f,e}\).
\end{proof}

\begin{corollary}[Comparable splittings]\label{cor:comparable}
If \(e,f\in\Ecal(R)\) and \(e\leq f\), then
\[
T_e\cong (f-e)R\times T_f,
\]
where \((f-e)R\) is a finite product of division rings.
\end{corollary}

\begin{proof}
Here \(e(1-f)=0\), \(f(1-e)=f-e\), and
\((1-e)(1-f)=1-f\).  Apply Theorem~\ref{thm:comparison}.
\end{proof}

\begin{theorem}[Canonical residual]
\label{thm:canonical}
Assume that \(\Idem Z(R)\) satisfies the ascending chain condition.
Then \(\Ecal(R)\) has a unique greatest element \(e_{\max}\).  The
factor
\[
T_{\mathrm{can}}=(1-e_{\max})R
\]
is therefore canonical.  Moreover:
\begin{enumerate}
\item \(e_{\max}R\) contains every central semisimple artinian direct
      factor of \(R\);
\item \(T_{\mathrm{can}}\) has no nonzero central semisimple artinian
      direct factor;
\item for every \(e\in\Ecal(R)\),
\[
T_e\cong D_e\times T_{\mathrm{can}}
\]
      for a finite product \(D_e\) of division rings.
\end{enumerate}
\end{theorem}

\begin{proof}
The set \(\Ecal(R)\) is nonempty and, by
Proposition~\ref{prop:join}, is closed under finite joins.  The ACC gives
a maximal member \(e_{\max}\).  If \(e\in\Ecal(R)\), then
\(e_{\max}\vee e\in\Ecal(R)\) and
\(e_{\max}\leq e_{\max}\vee e\).  Maximality forces
\(e_{\max}\vee e=e_{\max}\), hence \(e\leq e_{\max}\).
Thus \(e_{\max}\) is greatest and is unique.

Let \(c\in Z(R)\) be idempotent with \(cR\) semisimple artinian.
Proposition~\ref{prop:join} gives \(e_{\max}\vee c\in\Ecal(R)\), so
greatestness forces \(c\leq e_{\max}\).  This proves (1).  If the
canonical residual had a nonzero central semisimple artinian direct
factor with identity \(d\), then \(d\) would also be central in \(R\),
\(d\leq1-e_{\max}\), and (1) would give \(d\leq e_{\max}\), a
contradiction.  This proves (2), while (3) is an application of
Corollary~\ref{cor:comparable}.
\end{proof}

\begin{example}[Why a finiteness hypothesis is needed]
\label{ex:infinite-product}
Let \(I\) be infinite and let
\[
R=\prod_{i\in I}F_i
\]
for fields \(F_i\).  If ideals \(A,B\) have zero intersection, their
coordinate supports are disjoint: if coordinate \(i\) occurred in both,
multiplication by the \(i\)-th coordinate idempotent and by suitable
scalars would put that idempotent in \(A\cap B\).  For an
\(R\)-homomorphism \(\varphi:A\to B\) and \(a\in A\), multiplication by
the characteristic function of the support of \(a\) gives
\(\varphi(a)=0\).  Thus \(\Hom_R(A,B)=0\).  Hence \(R_R\) is square-free
and \(R\) is strongly right \(\Cfourstar\).

Central idempotents are characteristic functions \(e_A\) of subsets
\(A\subseteq I\).  The factor \(e_AR\) is semisimple artinian exactly
when \(A\) is finite (an infinite product has a strict descending chain
of coordinate ideals), while \((1-e_A)R\) is again
summand-square-free.  Thus
\[
\Ecal(R)=\{e_A:A\subseteq I\text{ finite}\}.
\]
It is closed under finite joins but has no greatest member.  Therefore a
canonical largest semisimple factor cannot be obtained without an
additional finiteness condition such as the central ACC in
Theorem~\ref{thm:canonical}.
\end{example}

We now remove the semisimple layer from the transfer problem.  The
conclusion is independent of the admissible idempotent.

\begin{theorem}[Residual reduction and defect invariance]
\label{thm:residual-reduction}
For every \(e\in\Ecal(R)\):
\begin{enumerate}
\item \(R\) is a left \(\Cfourstar\)-ring if and only if \(T_e\) is a
      left \(\Cfourstar\)-ring;
\item \(R\) is strongly left \(\Cfourstar\) if and only if \(T_e\) is
      strongly left \(\Cfourstar\).
\end{enumerate}
In particular, the truth or failure of either left-sided condition is
the same for all admissible residual factors.
\end{theorem}

\begin{proof}
The ring \(\Sigma_e=eR\) is semisimple artinian, hence
\(\Cfourstar\) and semi-weak-CS on both sides.  Since
\[
R\cong\Sigma_e\times T_e,
\]
the assertions follow from Lemma~\ref{lem:product}.
\end{proof}

\begin{corollary}\label{cor:regular-recovery}
If \(T_e\) is regular for some \(e\in\Ecal(R)\), then \(R\) is strongly
left \(\Cfourstar\).
\end{corollary}

\begin{proof}
A regular right summand-square-free ring is strongly regular
\cite[Lemma~4.7]{IbrahimEidElGuindy2026}.  Hence \(T_e\) is square-free
on both sides and is strongly left \(\Cfourstar\) by
Proposition~\ref{prop:squarefree-strong}.  Apply
Theorem~\ref{thm:residual-reduction}.
\end{proof}

\section{Anti-isomorphism transport}\label{sec:anti}

Opposite-ring techniques and involutions are standard in ring theory
\cite{Lam2001,KnusMerkurjevRostTignol1998,Pierce1982}.  What is needed
here is an exact transport statement for the strong
\(\Cfourstar\) structure.

\begin{definition}
Let \(\alpha:R\to S\) be an anti-isomorphism.  For a right \(R\)-module
\(M\), let \({}^{\alpha}M\) denote the left \(S\)-module with the same
additive group as \(M\) and action
\[
s\cdot m=m\,\alpha^{-1}(s)
\qquad(s\in S,\;m\in M).
\]
\end{definition}

\begin{lemma}[Lattice transport]\label{lem:lattice-transport}
The correspondence \(M\mapsto{}^{\alpha}M\) has the following
properties.
\begin{enumerate}
\item The right \(R\)-submodules of \(M\) are exactly the left
      \(S\)-submodules of \({}^{\alpha}M\), as additive subgroups.
\item A map between right \(R\)-modules is \(R\)-linear if and only if it
      is \(S\)-linear between the transported left modules.
\item Kernels, images, intersections, sums, direct summands, essential
      submodules, semisimple submodules, and chains of submodules are
      unchanged by transport.
\end{enumerate}
\end{lemma}

\begin{proof}
Surjectivity of \(\alpha\) shows that closure under the right \(R\)-action
is equivalent to closure under the transported left \(S\)-action.  For a
homomorphism \(f\),
\[
f(s\cdot m)
=f(m\alpha^{-1}(s))
=f(m)\alpha^{-1}(s)
=s\cdot f(m).
\]
This proves the first two assertions.  The third follows because the
underlying subgroups and homomorphisms are identical.
\end{proof}

\begin{theorem}[Anti-isomorphism transport theorem]
\label{thm:anti-transport}
Let \(\alpha:R\to S\) be an anti-isomorphism and let \(M\) be a right
\(R\)-module.  The following equivalences hold:
\begin{align*}
M\text{ is }\Cfour
&\Longleftrightarrow {}^{\alpha}M\text{ is }\Cfour,\\
M\text{ is }\Cfourstar
&\Longleftrightarrow {}^{\alpha}M\text{ is }\Cfourstar,\\
M\text{ is semi-weak-CS}
&\Longleftrightarrow {}^{\alpha}M\text{ is semi-weak-CS},\\
M\text{ is strongly }\Cfourstar
&\Longleftrightarrow {}^{\alpha}M\text{ is strongly }\Cfourstar.
\end{align*}
\end{theorem}

\begin{proof}
The \(\Cfour\) condition is formulated using a decomposition, a
homomorphism, its kernel, its image, and the direct-summand relation.
All of these are unchanged by Lemma~\ref{lem:lattice-transport}.  The
\(\Cfour\) equivalence follows.  Applying it to every submodule gives
the \(\Cfourstar\) equivalence.

The set \(\mathcal S(M)\), its order, and the unions of its chains are
also unchanged.  Semisimplicity, essentiality, and the direct-summand
relation are preserved.  Hence the semi-weak-CS conditions are
equivalent.  Combining the last two equivalences proves the strong
statement.
\end{proof}

\begin{corollary}[Two-ring transfer]\label{cor:two-ring}
If \(\alpha:R\to S\) is an anti-isomorphism, then
\[
R\text{ is strongly right }\Cfourstar
\quad\Longleftrightarrow\quad
S\text{ is strongly left }\Cfourstar.
\]
The same assertion holds with ``strongly \(\Cfourstar\)'' replaced by
``\(\Cfourstar\)''.
\end{corollary}

\begin{proof}
The map
\[
{}^{\alpha}(R_R)\longrightarrow {}_SS,\qquad r\longmapsto\alpha(r),
\]
is an isomorphism of left \(S\)-modules.  Apply
Theorem~\ref{thm:anti-transport}.
\end{proof}

\begin{corollary}[Self-opposite symmetry]\label{cor:self-opposite}
If \(R\cong R^{\op}\), then the right and left \(\Cfourstar\) conditions
are equivalent, and the strongly right and strongly left
\(\Cfourstar\) conditions are equivalent.
\end{corollary}

\begin{remark}
An involution is more than is needed in
Corollary~\ref{cor:self-opposite}.  Any anti-automorphism suffices; it
need not have order two.
\end{remark}

\section{Transfer through the residual factor}\label{sec:transfer}

We combine comparison with side-changing transport.  The semisimple
summand-square-free factors isolated by
Theorem~\ref{thm:comparison} are harmless on both sides, even when they
are not self-opposite.

\begin{definition}
A ring \(T\) is \emph{division-stably self-opposite} if
\[
T\cong D\times U,
\]
where \(D\) is a finite product of division rings and
\(U\cong U^{\op}\).  The zero ring is allowed as either factor.
\end{definition}

\begin{theorem}[Division-stable right-to-left transfer]
\label{thm:main-transfer}
Let \(R\) be a strongly right \(\Cfourstar\)-ring.  If \(T_e\) is
division-stably self-opposite for some \(e\in\Ecal(R)\), then \(R\) is
strongly left \(\Cfourstar\).
\end{theorem}

\begin{proof}
Write \(T_e\cong D\times U\), where \(D\) is a finite product of
division rings and \(U\cong U^{\op}\).  By
Theorem~\ref{thm:central-splitting} and Lemma~\ref{lem:product}, both
\(D\) and \(U\) are strongly right \(\Cfourstar\).  The semisimple
artinian ring \(D\) is strongly \(\Cfourstar\) on both sides, while
Corollary~\ref{cor:self-opposite} makes \(U\) strongly left
\(\Cfourstar\).  The product lemma makes \(T_e\) strongly left, and
Theorem~\ref{thm:residual-reduction} gives the conclusion for \(R\).
\end{proof}

\begin{corollary}[Choice independence]\label{cor:choice-independent}
If one admissible residual \(T_e\) is division-stably self-opposite, then
every admissible residual gives right-to-left transfer for \(R\).  Under
the central ACC, right-to-left transfer follows in particular when the
canonical residual \(T_{\mathrm{can}}\) is division-stably
self-opposite.
\end{corollary}

\begin{proof}
The first assertion follows from Theorems~\ref{thm:main-transfer} and
\ref{thm:residual-reduction}.  The final assertion is
Theorem~\ref{thm:main-transfer} applied to \(e_{\max}\).
\end{proof}

\begin{corollary}\label{cor:involutive-residual}
Right-to-left transfer holds if, for some \(e\in\Ecal(R)\), the factor
\(T_e\) is self-opposite, commutative, or admits an anti-automorphism.
\end{corollary}

\begin{remark}
Theorem~\ref{thm:main-transfer} does not require
\(D\cong D^{\op}\).  This is the point of retaining the division-ring
ambiguity in Theorem~\ref{thm:comparison} instead of imposing uniqueness
of the residual factor.  The result is also distinct from covariant
Morita invariance: it transports a right module structure to a left
module structure through an anti-isomorphism
\cite{GokavarapuMorita2026}.
\end{remark}

\section{Nonregular self-opposite residual families}
\label{sec:examples}

The preceding transfer theorem would add nothing to the regular theory
unless its hypotheses occurred in nonregular rings.  We now give a broad
source of such rings.

\begin{lemma}\label{lem:uniform-squarefree}
Every uniform module is square-free.  Consequently, every right uniform
ring is strongly right \(\Cfourstar\), and every left uniform ring is
strongly left \(\Cfourstar\).
\end{lemma}

\begin{proof}
Two nonzero submodules of a uniform module cannot have zero intersection.
Thus the module is square-free.  Apply
Proposition~\ref{prop:squarefree-strong}.
\end{proof}

\begin{proposition}[Self-opposite Ore domains]\label{prop:ore-domain}
Let \(D\) be a right Ore domain such that \(D\cong D^{\op}\).  Then \(D\)
is strongly \(\Cfourstar\) on both sides.  If \(D\) is not a division
ring, then \(D\) is not von Neumann regular.
\end{proposition}

\begin{proof}
Let \(I,J\) be nonzero right ideals.  Choose \(0\neq a\in I\) and
\(0\neq b\in J\).  The right Ore condition gives nonzero \(u,v\in D\)
such that
\[
au=bv\neq0.
\]
Thus \(I\cap J\neq0\), and \(D_D\) is uniform.  It is strongly right
\(\Cfourstar\) by Lemma~\ref{lem:uniform-squarefree}.  Self-oppositeness
and Corollary~\ref{cor:self-opposite} give the left-handed conclusion.

A von Neumann regular domain is a division ring.  Indeed, if
\(0\neq a=aba\), cancellation gives \(ba=1\), and then
\((ab-1)a=0\) gives \(ab=1\).  Hence a nondivision domain cannot be
regular.
\end{proof}

We give a noncommutative family.  The standard facts about skew Laurent
extensions used in the proof are recalled in
\cite{GoodearlWarfield2004,McConnellRobson2001}.

\begin{theorem}[Skew Laurent family]\label{thm:skew-laurent}
Let \(K\) be a field and let \(\sigma\in\operatorname{Aut}(K)\).  Put
\[
L=K[x,x^{-1};\sigma],
\qquad xa=\sigma(a)x\quad(a\in K).
\]
Then:
\begin{enumerate}
\item \(L\) admits the involution determined by
      \(a^\dagger=a\) for \(a\in K\) and \(x^\dagger=x^{-1}\);
\item \(L\) is strongly \(\Cfourstar\) on both sides;
\item \(L\) is not von Neumann regular.
\end{enumerate}
If \(\sigma\neq\id_K\), then \(L\) is noncommutative.
\end{theorem}

\begin{proof}
The defining relation is preserved anti-multiplicatively because
\[
(xa)^\dagger=ax^{-1}
\quad\text{and}\quad
(\sigma(a)x)^\dagger=x^{-1}\sigma(a)=ax^{-1}.
\]
Thus the indicated rule extends to an anti-automorphism of \(L\), and it
squares to the identity on the generators.

The ring \(L\) is a left and right noetherian domain.  Hence it is a left
and right Ore domain.  Proposition~\ref{prop:ore-domain} gives the
strong \(\Cfourstar\) property on both sides.

It remains to exclude regularity.  For a nonzero Laurent polynomial
\(f\), let \(w(f)\) be the difference between its largest and smallest
exponents.  The domain property and the automorphism hypothesis give
\[
w(fg)=w(f)+w(g).
\]
Every unit therefore has width zero and is of the form \(ax^n\) with
\(a\in K^\times\).  The element \(1+x\) is not a unit.  Hence \(L\) is
not a division ring and, by Proposition~\ref{prop:ore-domain}, is not
regular.  Noncommutativity for nontrivial \(\sigma\) follows from
\(xa=\sigma(a)x\).
\end{proof}

\begin{corollary}[Nonregular central products]\label{cor:nonregular-product}
Let \(\Sigma\) be any semisimple artinian ring and let \(L\) be as in
Theorem~\ref{thm:skew-laurent}.  Then
\[
R=\Sigma\times L
\]
is a nonregular ring which is strongly \(\Cfourstar\) on both sides.
\end{corollary}

\begin{proof}
The factors are strongly \(\Cfourstar\) on both sides, so the product
lemma applies.  Since the direct factor \(L\) is not regular, neither is
\(R\).
\end{proof}

\begin{example}[Movable division layers]\label{ex:movable-layers}
Let \(\Sigma\) be semisimple artinian, let
\[
D=D_1\times\cdots\times D_n
\]
be a finite product of division rings, and let
\(L=K[x,x^{-1};\sigma]\) be as in
Theorem~\ref{thm:skew-laurent}.  In
\[
R=\Sigma\times D\times L
\]
let \(e\) be the identity of \(\Sigma\) and let \(f\) be the identity of
\(\Sigma\times D\).  Then \(e,f\in\Ecal(R)\), \(e\leq f\), and
\[
T_e\cong D\times L,\qquad T_f\cong L.
\]
Indeed, \(D_D\) is summand-square-free by
Lemma~\ref{lem:ssf-semisimple}, and \(L_L\) is square-free by
Lemma~\ref{lem:uniform-squarefree}.  Direct summands and homomorphisms
over \(D\times L\) are computed componentwise, so
\((D\times L)_{D\times L}\) is summand-square-free.
Thus a division-ring layer may be placed either in the semisimple factor
or in the residual factor.  Corollary~\ref{cor:comparable} shows that
this is the only type of ambiguity possible.  If \(L\) is nonregular,
then so is \(R\), while \(R\) is strongly \(\Cfourstar\) on both sides.
\end{example}

\begin{example}
Let \(k\) be a field and take
\[
K=k(t_i\mid i\in\mathbb Z),
\qquad \sigma(t_i)=t_{i+1}.
\]
Then \(\sigma\) is a nontrivial automorphism of \(K\), and
\(K[x,x^{-1};\sigma]\) in
Theorem~\ref{thm:skew-laurent} is a concrete noncommutative,
nonregular, two-sided strongly \(\Cfourstar\)-ring.  This family lies
strictly outside the regular boundary of
\cite[Proposition~4.8]{IbrahimEidElGuindy2026}.
\end{example}

\section{Exact reduction of one-sided defects}
\label{sec:obstruction}

We now show that central splitting and opposite-ring transport locate
every possible failure of right-to-left transfer.

\begin{theorem}[Residual-core theorem]\label{thm:residual-core}
Let \(R\) be strongly right \(\Cfourstar\) but not strongly left
\(\Cfourstar\).  For every \(e\in\Ecal(R)\), the residual factor
\(T_e\) satisfies:
\begin{enumerate}
\item \((T_e)_{T_e}\) is summand-square-free;
\item \(T_e\) is strongly right \(\Cfourstar\);
\item \(T_e\) is not strongly left \(\Cfourstar\);
\item \(T_e\not\cong T_e^{\op}\);
\item \(T_e\) is not division-stably self-opposite.
\end{enumerate}
If \(e,f\in\Ecal(R)\), the two defective residuals have the common factor
\(C_{e,f}\) of Theorem~\ref{thm:comparison}, and their other factors are
finite products of division rings.  Conversely, if a ring \(T\) is
strongly right but not strongly left \(\Cfourstar\), then
\(\Sigma\times T\) has the same one-sided defect for every semisimple
artinian ring \(\Sigma\).
\end{theorem}

\begin{proof}
Item (1) is part of the definition of \(\Ecal(R)\), and item (2) follows
from \(R\cong eR\times T_e\) and Lemma~\ref{lem:product}.  If \(T_e\)
were strongly left
\(\Cfourstar\), then \(R\) would be strongly left by
Theorem~\ref{thm:residual-reduction}; this proves (3).  If
\(T_e\cong T_e^{\op}\), then (2) and
Corollary~\ref{cor:self-opposite} would contradict (3).  Thus (4) holds.
Item (5) follows directly from
Theorem~\ref{thm:main-transfer}.  The common-factor statement is
Theorem~\ref{thm:comparison}.

For the converse, the semisimple ring \(\Sigma\) is strongly
\(\Cfourstar\) on both sides.  Lemma~\ref{lem:product} shows that
\(\Sigma\times T\) is strongly right but not strongly left.
\end{proof}

\begin{corollary}[Reduction principle]\label{cor:reduction-principle}
The classification of strongly right \(\Cfourstar\)-rings which are not
strongly left \(\Cfourstar\) reduces exactly to the same classification
among right summand-square-free rings.  Semisimple direct factors neither
create nor remove the defect.
\end{corollary}

The known skew-polynomial separation shows that the residual-core theorem
is not vacuous.

\begin{proposition}[Sharp one-sided residual]\label{prop:sharp-residual}
Let \(K\) be a field admitting an injective nonsurjective endomorphism
\(\sigma\), and put
\[
A=K[x;\sigma].
\]
For the orientation in which
\cite[Example~4.12]{IbrahimEidElGuindy2026} makes \(A\) strongly left
\(\Cfourstar\) but not strongly right \(\Cfourstar\), the opposite ring
\[
R=A^{\op}
\]
is strongly right but not strongly left \(\Cfourstar\).  Moreover,
\(R_R\) is summand-square-free, \(R\not\cong R^{\op}\), and \(R\) is not
division-stably self-opposite.
\end{proposition}

\begin{proof}
Passing to the opposite ring interchanges the two sides, so the first
assertion follows from the cited example.  The skew polynomial ring
\(A\) is a domain.  Hence \(R\) has no nontrivial idempotents.  Apply
Theorem~\ref{thm:central-splitting} to \(R\).  Its central semisimple
factor is either \(0\) or \(R\).  The second possibility would make
\(R\) semisimple artinian, which it is not.  Therefore the residual
factor is all of \(R\), and \(R_R\) is summand-square-free.  Finally,
self-oppositeness would force left--right symmetry by
Corollary~\ref{cor:self-opposite}, a contradiction.  Since \(R\) is a
nonsemisimple domain, it has no nontrivial product decomposition and no
nonzero semisimple artinian direct factor.  Division-stable
self-oppositeness would therefore reduce to self-oppositeness, which has
just been excluded.
\end{proof}

\begin{remark}
The contrast between Theorem~\ref{thm:skew-laurent} and
Proposition~\ref{prop:sharp-residual} is structural.  When \(\sigma\) is
an automorphism and \(x\) is inverted, the skew Laurent ring possesses
an involution and the two sides agree.  When \(\sigma\) is injective but
not surjective in the one-sided skew polynomial construction, the known
example loses that opposite-ring symmetry and supports a one-sided
defect.
\end{remark}

\section{Conclusion}

The semisimple summand in the strong \(\Cfourstar\) decomposition of
\(R_R\) is not merely a module-theoretic summand.  Projectivity of the
regular module supplies Hom-vanishing in both directions.  The associated
cross-corners vanish, the defining idempotent is central, and the
decomposition is a ring product.

The new comparison theorem determines the nonuniqueness of this product:
two residual factors share a central core and differ only by finite
products of division rings.  Hence the residual transfer defect is an
intrinsic invariant even when no preferred decomposition exists.  Under
ACC on central idempotents, the admissible splittings have a greatest
member and give a canonical residual.  The infinite product example
marks the precise point at which this canonicality can fail.

Anti-isomorphism transport then yields a nonregular symmetry principle.
A residual which is self-opposite after removal of a finite
division-ring factor cannot support a one-sided strong
\(\Cfourstar\) defect.  The skew Laurent examples show that the criterion
contains noncommutative nonregular rings, whereas the opposite of the
known skew-polynomial example shows that genuine one-sided residuals
remain.

The remaining problem is now sharply localized: classify the
summand-square-free strongly right \(\Cfourstar\)-rings that are not
division-stably self-opposite and determine which fail the left strong
condition.  Semisimple factors and all ambiguity coming from division
rings have been separated from that question.

\section*{Acknowledgment}

The authors thank the Commissioner of Higher Education, Government of
Andhra Pradesh, India, and the Principal, Government College
(Autonomous), Rajahmundry, for providing a supportive academic
environment.

\section*{Declarations}

\noindent\textbf{Author contributions.}
All authors contributed to the mathematical discussions, verification,
preparation, and revision of the manuscript.  All authors reviewed and
approved the final version.\par

\smallskip
\noindent\textbf{Funding.}
The authors received no financial support for this research.\par

\smallskip
\noindent\textbf{Conflict of interest.}
The authors declare no conflict of interest.\par

\smallskip
\noindent\textbf{Data availability.}
No datasets were generated or analysed; the article is purely
theoretical.\par

\section*{ORCID}

\noindent Chandrasekhar Gokavarapu -
\url{https://orcid.org/0009-0006-5306-371X}

\bibliographystyle{amsplain}
\bibliography{references}

@article{AltunIbrahimOzcanYousif2018,
  author  = {Altun-{\"O}zarslan, Meltem and Ibrahim, Yasser and {\"O}zcan, A. {\c C}. and Yousif, Mohamed},
  title   = {{$C4$}- and {$D4$}-modules via perspective direct summands},
  journal = {Communications in Algebra},
  volume  = {46},
  number  = {10},
  pages   = {4480--4497},
  year    = {2018},
  doi     = {10.1080/00927872.2018.1448838},
  note    = {\url{https://doi.org/10.1080/00927872.2018.1448838}}
}

@article{DingIbrahimYousifZhou2017,
  author  = {Ding, Nanqing and Ibrahim, Yasser and Yousif, Mohamed and Zhou, Yiqiang},
  title   = {{$C4$}-modules},
  journal = {Communications in Algebra},
  volume  = {45},
  number  = {4},
  pages   = {1727--1740},
  year    = {2017},
  doi     = {10.1080/00927872.2016.1222412},
  note    = {\url{https://doi.org/10.1080/00927872.2016.1222412}}
}

@article{CamilloYu1994,
  author  = {Camillo, Victor P. and Yu, Hua-Ping},
  title   = {Exchange rings, units and idempotents},
  journal = {Communications in Algebra},
  volume  = {22},
  number  = {12},
  pages   = {4737--4749},
  year    = {1994},
  doi     = {10.1080/00927879408825098},
  note    = {\url{https://doi.org/10.1080/00927879408825098}}
}

@article{CrawleyJonsson1964,
  author  = {Crawley, Peter and J{\'o}nsson, Bjarni},
  title   = {Refinements for infinite direct decompositions of algebraic systems},
  journal = {Pacific Journal of Mathematics},
  volume  = {14},
  number  = {3},
  pages   = {797--855},
  year    = {1964},
  doi     = {10.2140/pjm.1964.14.797},
  note    = {\url{https://doi.org/10.2140/pjm.1964.14.797}}
}

@incollection{DialloDiopKourkiTribak2025,
  author    = {Diallo, Abdoul Djibril and Diop, Papa Cheikhou and Kourki, Farid and Tribak, Rachid},
  title     = {On a Generalization of {$C_4$}-Modules},
  booktitle = {Algebra and Its Applications},
  editor    = {Patel, Manoj Kumar and Ashraf, Mohammad and Mahdou, Najib and Kim, Hwankoo},
  series    = {Springer Proceedings in Mathematics \& Statistics},
  volume    = {474},
  pages     = {387--404},
  publisher = {Springer},
  address   = {Singapore},
  year      = {2025},
  doi       = {10.1007/978-981-97-6798-4_30},
  note      = {\url{https://doi.org/10.1007/978-981-97-6798-4_30}}
}

@article{DingIbrahimYousifZhouD42017,
  author  = {Ding, Nanqing and Ibrahim, Yasser and Yousif, Mohamed and Zhou, Yiqiang},
  title   = {{$D4$}-modules},
  journal = {Journal of Algebra and Its Applications},
  volume  = {16},
  number  = {9},
  pages   = {1750166},
  year    = {2017},
  doi     = {10.1142/S0219498817501663},
  note    = {\url{https://doi.org/10.1142/S0219498817501663}}
}

@article{EbrahimiAtaniKhoramdelHesari2016,
  author  = {Ebrahimi Atani, Shahabaddin and Khoramdel, Mehdi and Dolati Pish Hesari, Saboura},
  title   = {{$C3$}-Modules},
  journal = {Demonstratio Mathematica},
  volume  = {49},
  number  = {3},
  pages   = {282--292},
  year    = {2016},
  doi     = {10.1515/dema-2016-0024},
  note    = {\url{https://doi.org/10.1515/dema-2016-0024}}
}

@book{Facchini1998,
  author    = {Facchini, Alberto},
  title     = {Module Theory: Endomorphism Rings and Direct Sum Decompositions in Some Classes of Modules},
  series    = {Progress in Mathematics},
  volume    = {167},
  publisher = {Birkh{\"a}user},
  address   = {Basel},
  year      = {1998},
  doi       = {10.1007/978-3-0348-8774-8},
  note      = {\url{https://doi.org/10.1007/978-3-0348-8774-8}}
}

@article{Facchini2012,
  author  = {Facchini, Alberto},
  title   = {Direct-sum decompositions of modules with semilocal endomorphism rings},
  journal = {Bulletin of Mathematical Sciences},
  volume  = {2},
  number  = {2},
  pages   = {225--279},
  year    = {2012},
  doi     = {10.1007/s13373-012-0024-9},
  note    = {\url{https://doi.org/10.1007/s13373-012-0024-9}}
}

@article{GargGroverKhurana2014,
  author  = {Garg, Shelly and Grover, Harpreet K. and Khurana, Dinesh},
  title   = {Perspective rings},
  journal = {Journal of Algebra},
  volume  = {415},
  pages   = {1--12},
  year    = {2014},
  doi     = {10.1016/j.jalgebra.2013.09.055},
  note    = {\url{https://doi.org/10.1016/j.jalgebra.2013.09.055}}
}

@book{GoodearlWarfield2004,
  author    = {Goodearl, K. R. and Warfield, Jr., Robert B.},
  title     = {An Introduction to Noncommutative Noetherian Rings},
  edition   = {Second},
  series    = {London Mathematical Society Student Texts},
  volume    = {61},
  publisher = {Cambridge University Press},
  address   = {Cambridge},
  year      = {2004},
  doi       = {10.1017/CBO9780511841699},
  note      = {\url{https://doi.org/10.1017/CBO9780511841699}}
}

@article{GuilAsensioSrivastava2013,
  author  = {Guil Asensio, Pedro A. and Srivastava, Ashish K.},
  title   = {Automorphism-invariant modules satisfy the exchange property},
  journal = {Journal of Algebra},
  volume  = {388},
  pages   = {101--106},
  year    = {2013},
  doi     = {10.1016/j.jalgebra.2013.05.003},
  note    = {\url{https://doi.org/10.1016/j.jalgebra.2013.05.003}}
}

@misc{GokavarapuMorita2026,
  author        = {Gokavarapu, Chandrasekhar},
  title         = {Morita Invariance, Categorical Obstructions, and
                   Dimension Transfer for {$C4$}, {$C4^{\ast}$},
                   Strongly {$C4^{\ast}$}, and Semi-Weak-CS Modules},
  year          = {2026},
  eprint        = {2604.16326},
  archiveprefix = {arXiv},
  primaryclass  = {math.RA},
  doi           = {10.48550/arXiv.2604.16326},
  note          = {arXiv:2604.16326,
                   \url{https://doi.org/10.48550/arXiv.2604.16326}}
}

@misc{GokavarapuTransfer2026,
  author        = {Gokavarapu, Chandrasekhar},
  title         = {Left--right Transfer for {$C4^{\ast}$}-Rings Beyond
                   the Regular Case},
  year          = {2026},
  eprint        = {2605.04053},
  archiveprefix = {arXiv},
  primaryclass  = {math.RA},
  doi           = {10.48550/arXiv.2605.04053},
  note          = {arXiv:2605.04053, version 1; superseded by the present
                   manuscript,
                   \url{https://doi.org/10.48550/arXiv.2605.04053}}
}

@article{IbrahimEidElGuindy2026,
  author  = {Ibrahim, Yasser and Eid, Hussein and El-Guindy, Ahmad},
  title   = {On {$C4$}-modules},
  journal = {Journal of Algebra and Its Applications},
  volume  = {25},
  number  = {7},
  pages   = {2650061},
  year    = {2026},
  doi     = {10.1142/S0219498826500611},
  note    = {\url{https://doi.org/10.1142/S0219498826500611}}
}

@article{IbrahimYousif2018,
  author  = {Ibrahim, Yasser and Yousif, Mohamed},
  title   = {Rings all of whose right ideals are {$U$}-modules},
  journal = {Communications in Algebra},
  volume  = {46},
  number  = {5},
  pages   = {1983--1995},
  year    = {2018},
  doi     = {10.1080/00927872.2017.1365881},
  note    = {\url{https://doi.org/10.1080/00927872.2017.1365881}}
}

@article{IbrahimYousif2019,
  author  = {Ibrahim, Yasser and Yousif, Mohamed},
  title   = {Dual-square-free modules},
  journal = {Communications in Algebra},
  volume  = {47},
  number  = {7},
  pages   = {2954--2966},
  year    = {2019},
  doi     = {10.1080/00927872.2018.1543429},
  note    = {\url{https://doi.org/10.1080/00927872.2018.1543429}}
}

@article{IbrahimYousif2022,
  author  = {Ibrahim, Yasser and Yousif, Mohamed},
  title   = {{$C4$}-modules with the exchange property},
  journal = {Communications in Algebra},
  volume  = {50},
  number  = {12},
  pages   = {5435--5443},
  year    = {2022},
  doi     = {10.1080/00927872.2022.2085289},
  note    = {\url{https://doi.org/10.1080/00927872.2022.2085289}}
}

@article{IbrahimYousifPerspective2022,
  author  = {Ibrahim, Yasser and Yousif, Mohamed},
  title   = {Perspectivity, exchange and summand-dual-square-free modules},
  journal = {Communications in Algebra},
  volume  = {50},
  number  = {6},
  pages   = {2488--2506},
  year    = {2022},
  doi     = {10.1080/00927872.2021.2008414},
  note    = {\url{https://doi.org/10.1080/00927872.2021.2008414}}
}

@article{IbrahimYousifDirectComplements2023,
  author  = {Ibrahim, Yasser and Yousif, Mohamed},
  title   = {Direct complements almost unique},
  journal = {Journal of Algebra and Its Applications},
  volume  = {22},
  number  = {12},
  pages   = {2350260},
  year    = {2023},
  doi     = {10.1142/S0219498823502602},
  note    = {\url{https://doi.org/10.1142/S0219498823502602}}
}

@article{KosanQuynhSrivastava2016,
  author  = {Ko{\c s}an, M. Tamer and Quynh, Truong Cong and Srivastava, Ashish K.},
  title   = {Rings with each right ideal automorphism-invariant},
  journal = {Journal of Pure and Applied Algebra},
  volume  = {220},
  number  = {4},
  pages   = {1525--1537},
  year    = {2016},
  doi     = {10.1016/j.jpaa.2015.09.016},
  note    = {\url{https://doi.org/10.1016/j.jpaa.2015.09.016}}
}

@book{KnusMerkurjevRostTignol1998,
  author    = {Knus, Max-Albert and Merkurjev, Alexander and Rost, Markus and Tignol, Jean-Pierre},
  title     = {The Book of Involutions},
  series    = {American Mathematical Society Colloquium Publications},
  volume    = {44},
  publisher = {American Mathematical Society},
  address   = {Providence, RI},
  year      = {1998},
  doi       = {10.1090/coll/044},
  note      = {\url{https://doi.org/10.1090/coll/044}}
}

@book{Lam1999,
  author    = {Lam, T. Y.},
  title     = {Lectures on Modules and Rings},
  series    = {Graduate Texts in Mathematics},
  volume    = {189},
  publisher = {Springer},
  address   = {New York},
  year      = {1999},
  doi       = {10.1007/978-1-4612-0525-8},
  note      = {\url{https://doi.org/10.1007/978-1-4612-0525-8}}
}

@book{Lam2001,
  author    = {Lam, T. Y.},
  title     = {A First Course in Noncommutative Rings},
  edition   = {Second},
  series    = {Graduate Texts in Mathematics},
  volume    = {131},
  publisher = {Springer},
  address   = {New York},
  year      = {2001},
  doi       = {10.1007/978-1-4419-8616-0},
  note      = {\url{https://doi.org/10.1007/978-1-4419-8616-0}}
}

@article{MazurekNielsenZiembowski2015,
  author  = {Mazurek, Ryszard and Nielsen, Pace P. and Ziembowski, Micha{\l}},
  title   = {Commuting idempotents, square-free modules, and the exchange property},
  journal = {Journal of Algebra},
  volume  = {444},
  pages   = {52--80},
  year    = {2015},
  doi     = {10.1016/j.jalgebra.2015.07.015},
  note    = {\url{https://doi.org/10.1016/j.jalgebra.2015.07.015}}
}

@book{MohamedMuller1990,
  author    = {Mohamed, Saad H. and M{\"u}ller, Bruno J.},
  title     = {Continuous and Discrete Modules},
  series    = {London Mathematical Society Lecture Note Series},
  volume    = {147},
  publisher = {Cambridge University Press},
  address   = {Cambridge},
  year      = {1990},
  doi       = {10.1017/CBO9780511600692},
  note      = {\url{https://doi.org/10.1017/CBO9780511600692}}
}

@book{McConnellRobson2001,
  author    = {McConnell, J. C. and Robson, J. C.},
  title     = {Noncommutative Noetherian Rings},
  edition   = {Revised},
  series    = {Graduate Studies in Mathematics},
  volume    = {30},
  publisher = {American Mathematical Society},
  address   = {Providence, RI},
  year      = {2001},
  doi       = {10.1090/gsm/030},
  note      = {\url{https://doi.org/10.1090/gsm/030}}
}

@article{Nielsen2010,
  author  = {Nielsen, Pace P.},
  title   = {Square-free modules with the exchange property},
  journal = {Journal of Algebra},
  volume  = {323},
  number  = {7},
  pages   = {1993--2001},
  year    = {2010},
  doi     = {10.1016/j.jalgebra.2009.12.035},
  note    = {\url{https://doi.org/10.1016/j.jalgebra.2009.12.035}}
}

@article{Nicholson1977,
  author  = {Nicholson, W. Keith},
  title   = {Lifting idempotents and exchange rings},
  journal = {Transactions of the American Mathematical Society},
  volume  = {229},
  pages   = {269--278},
  year    = {1977},
  doi     = {10.1090/S0002-9947-1977-0439876-2},
  note    = {\url{https://doi.org/10.1090/S0002-9947-1977-0439876-2}}
}

@article{Ore1933,
  author  = {Ore, {\O}ystein},
  title   = {Theory of non-commutative polynomials},
  journal = {Annals of Mathematics},
  series  = {2},
  volume  = {34},
  number  = {3},
  pages   = {480--508},
  year    = {1933},
  doi     = {10.2307/1968173},
  note    = {\url{https://doi.org/10.2307/1968173}}
}

@book{Pierce1982,
  author    = {Pierce, Richard S.},
  title     = {Associative Algebras},
  series    = {Graduate Texts in Mathematics},
  volume    = {88},
  publisher = {Springer},
  address   = {New York},
  year      = {1982},
  doi       = {10.1007/978-1-4757-0163-0},
  note      = {\url{https://doi.org/10.1007/978-1-4757-0163-0}}
}

@article{Warfield1972,
  author  = {Warfield, Jr., Robert B.},
  title   = {Exchange rings and decompositions of modules},
  journal = {Mathematische Annalen},
  volume  = {199},
  pages   = {31--36},
  year    = {1972},
  note    = {\url{https://eudml.org/doc/162311}}
}

@article{YousifAminIbrahim2014,
  author  = {Yousif, Mohamed and Amin, Ismail and Ibrahim, Yasser},
  title   = {{$D3$}-Modules},
  journal = {Communications in Algebra},
  volume  = {42},
  number  = {2},
  pages   = {578--592},
  year    = {2014},
  doi     = {10.1080/00927872.2012.718823},
  note    = {\url{https://doi.org/10.1080/00927872.2012.718823}}
}

@article{Zhu2023,
  author  = {Zhu, Zhanmin},
  title   = {Generalizations of {$C3$} modules and {$C4$} modules},
  journal = {Mathematical Reports},
  volume  = {25(75)},
  number  = {1},
  pages   = {187--197},
  year    = {2023},
  note    = {\url{https://imar.ro/journals/Mathematical_Reports/php/2023/Mrc23_1.php}}
}

\end{document}